\documentclass[oneside,english]{amsart}
\usepackage[T1]{fontenc}
\usepackage[utf8]{inputenc}
\usepackage{amsthm}
\usepackage{amstext}
\usepackage{amssymb}
\usepackage{amsmath}
\usepackage{mathtools}
\usepackage[english]{babel}

\newtheorem*{thm*}{Theorem}
\newtheorem{thm}{Theorem}[section]
\newtheorem{THM}{Theorem}

\newtheorem{prop}[thm]{Proposition}

\newtheorem{lemma}[thm]{Lemma}

\theoremstyle{definition}

\newtheorem{remark}[thm]{Remark}

\DeclareMathOperator{\diff}{Diff}

\title{Distributions and Legendrian foliations in dimension $3$}
\author{Maycol Falla Luza}
\author{Rudy Rosas}
\begin{document}
\maketitle

\begin{abstract}
We study the field of rational first integrals of distributions. We show that for a distribution on $3$ dimensional manifolds there exists a tangent vector field with the same field of first integrals. We also show a similar result for integrable distributions in any dimension.
\end{abstract}

\section{Introduction}

Let $M$ be a projective manifold. If $X$ is a rational vector field  on $M$, a rational first integral of $X$ is a rational function in $M$ that is  constant along the integral curves of $X$. We denote  by $\mathbb{C}(X)$ the field of rational first integrals of $X$.  Furthermore, if  $\mathcal D$ is a singular holomorphic distribution on $M$,  we say that $h$ is a rational first integral of $\mathcal{D}$ if $h \in \mathbb{C}(X)$ for all  rational  vector field $X$  tangent to $\mathcal{D}$. We denote by $\mathbb{C}(\mathcal{D})$ the field of rational  first integrals of $\mathcal{D}$. 

In the case of distributions on  affine manifolds,  the algebra $\mathbb{C}[\mathcal{D}]$ of polynomial first integrals has been studied by several authors, see for example \cite{Bon}, \cite{NagNow}, \cite{Now}.  For instance, Nowicki shows in \cite{Now} the existence of a vector field $X$ such that $\mathbb{C}[X]=\mathbb{C}[\mathcal{D}]$.
P. Bonnet refines Nowicki's result and shows in \cite{Bon} that $X$ can be found being tangent to $\mathcal{D}$. A natural question -- also posed by Bonnet in his paper --- is to find an analogous result for rational first integrals:  given a singular holomorphic distribution $\mathcal D$ on $M$, is it possible to find a rational vector field $X$ on $M$ such that 
 $\mathbb{C}(X)=\mathbb{C}(\mathcal{D})$?  The question  is quite addressable when the distribution is integrable, as we show as a first result of the paper.

\begin{THM}\label{Teorema 1} 
Let $M$ be a projective manifold and let $\mathcal D$ be  an  integrable   singular holomorphic distribution  on $M$. Then there is a rational vector field  $X$ on $M$ such that $\mathbb{C}(X)=\mathbb{C}(\mathcal{D})$.
\end{THM}

As the main result of the paper, we give a positive answer to the question in dimension 3. 

\begin{THM}\label{Teorema 2}
Let $M$ be a projective manifold of dimension $3$ and  let $\mathcal{D}$ be  a codimension one holomorphic distribution on $M$. Then there is a  rational vector field  $X$ on $M$ such that $\mathbb{C}(X)=\mathbb{C}(\mathcal{D})$.
\end{THM}

It is worth mentioning  that in \cite{Faro} it is proven a similar result in the case of  distributions with separated variables. 

In view of   Theorem \ref{Teorema 1}, the proof of  Theorem \ref{Teorema 2} will be  restricted to the case where  $\mathcal D$ is
not integrable. This happens, for example, when $\mathcal D$ is a contact structure. In this case, the leaves of $X$ are holomorphic Legendrian curves of $\mathcal D$. For this reason, given a distribution $\mathcal D$ in dimension three, the vector fields that are tangent to $\mathcal D$ will be called Legendrian vector fields. Analogously, the one-dimensional  singular holomorphic foliations that are tangent to $\mathcal D$ will be called Legendrian foliations.

The paper is organized as follows: 
 Theorem \ref{Teorema 1} is proved  in Section \ref{Section 2}. The remaining sections are devoted to the proof of
 Theorem \ref{Teorema 2}, for which    we assume that $\mathcal{D}$ is not integrable. 

\section{Proof of Theorem \ref{Teorema 1}}\label{Section 2}
Let  $p=\dim \mathcal D\ge 2$ and let $\mathcal G$ be the singular holomorphic foliation defined by $\mathcal D$. 
Take a generic point $x \in M$ and let $L$ be the leaf of $\mathcal{G}$ passing through $x$. Denote by $\mathfrak{M}_x$ the sheaf of sections of $\mathcal{O}_M$ vanishing at $x$, thus $T_{\mathcal{G}}\otimes \mathfrak{M}_x$ is the sheaf of vector fields tangent to $\mathcal{G}$ and vanishing at $x$. Taking local coordinates $(z_1, \ldots, z_n)$ around $x=0$ such that $\mathcal{G}$ is defined by $dz_{p+1}= \ldots = dz_n=0$ we see that $\operatorname{Hom}(T_xL, T_xL) \simeq M_p(\mathbb{C})$. Identifying this space with the skyscraper sheaf which has this vector space as its stalk at $x$, we have an exact sequence 
$$
0\rightarrow T_\mathcal{G}\otimes \mathfrak{M}_x^2 \rightarrow T_\mathcal{G}\otimes \mathfrak{M}_x \rightarrow M_p(\mathbb{C}) \rightarrow 0,
$$
where the image of the last map is obtained by taking the linear part of the local vector fields at the singular point $x$. Take now $\mathcal{L}$ an ample line bundle and $k>0$ so big that $H^1(M, T_\mathcal{G}\otimes \mathfrak{M}_x^2 \otimes \mathcal{L}^k)=0$. Therefore the associated map
$$
H^0(M, T_\mathcal{G}\otimes \mathfrak{M}_x \otimes \mathcal{L}^k) \rightarrow M_p(\mathbb{C})
$$
is surjective. Then we can choose a vector field $X$  tangent to $\mathcal{G}$ such that $X|_{L}$ has an isolated singularity at $x$ and a very generic linear part $$DX|_L (p)=(\mu_1 z_1,\ldots, \mu_d z_d, 0, \ldots, 0 ).$$ In particular $X|_L$ is locally linearizable  (see \cite[p. 183]{Ar}) and has no  non-constant meromorphic first integral. By taking $L_t$ a leaf of $\mathcal{G}$ close to $L$ we have the same for $X|_{L_t}$. So if $h \in \mathbb{C}(X)$ then $h|_{L_t}$ must be constant and so  $h\in \mathbb{C}(\mathcal{G})$.

\section{Lifting one-dimensional foliations}\label{sec1}
The Legendrian foliation announced in Theorem \ref{Teorema 2} will be constructed as a lifting of a suitable foliation $\mathcal{F}$ on $\mathbb{P}^2$, in the sense of the following proposition.

\begin{prop}\label{lifting}
Let $M$ and $N$ be projective manifolds of dimension $m$ and $n$ respectively and let $\Pi: M \dashrightarrow N$ be  a dominant rational map. Let $\mathcal{D}$ be a singular holomorphic distribution of dimension $n$ on $M$,  generically transverse to the fibers of $\Pi$. Then, given a one-dimensional singular holomorphic foliation $\mathcal{F}$ on $N$, there is a unique  Legendrian holomorphic foliation $\mathcal{F}^{\mathcal{D}}$ on $M$ such that $d\Pi(T_{\mathcal{F}^{\mathcal{D}}})=T_{\mathcal{F}}$.
\end{prop}
\begin{proof}
Let $p \in M$ be a generic point and let $q=\pi(p) \in N$. Then $$\dim d\Pi(p)^{-1}(T_{\mathcal{F}}(q)) = m-n+1$$  and so,  by transversality,  $\mathcal{D}(p)\cap d\Pi(p)^{-1}(T_{\mathcal{F}}(q))$ will be a one dimensional subspace of $T_pM$, defining our desired foliation.  
\end{proof}

From now on, we  focus in the case $m=3$, $N=\mathbb P ^2$. At a  generic point $\mathfrak m\in M$ we can find local holomorphic coordinates $(x,y,z)$,    such that $\Pi(x,y,z)=(x,y)$, where $(x,y)$ are affine coordinates on $\mathbb P^2$. Since $\mathfrak m$ is generic, we can assume that $\mathcal D$ is transverse to $\Pi^{-1}(0)$ at the origin, whence   $\mathcal{D}$ is defined near the origin by a holomorphic 1-form $$\Omega= dz - P(x,y,z)dx - Q(x,y,z)dy.$$  Up to a change of coordinates of the form 
$z\mapsto z+ax+by$, we have that $P(0)=Q(0)$. Moreover, if $\mathcal D$ is not integrable, we can suppose that $\Omega\wedge d\Omega\neq 0$ at $\mathfrak m$, which means that $Q_x(0)-P_y(0)\neq 0$. In this case, a further change of coordinates of the form $(x,y)\mapsto (\nu x,\nu y)$ allows us to assume that $Q_x(0)-P_y(0)=1$. 

A singular holomorphic foliation $\mathcal{F}$ on $\mathbb P^2$ is defined by a polinomial vector field  $$X=A(x,y)\frac{\partial}{\partial x}+B(x,y) \frac{\partial}{\partial y}.$$  Thus,   the lifting $\mathcal{F}^{\mathcal{D}}$  is  defined in the local coordinates $(x,y,z)$ by the holomorphic vector field $$X^{\mathcal{D}}=A\frac{\partial}{\partial x}+B\frac{\partial}{\partial y} +(AP+BQ)\frac{\partial}{\partial z}.$$

In the next section we will choose the foliation $\mathcal{F}$ on $\mathbb{P}^2$ that will be lifted in order to prove  our main result.

\section{Selection of the foliation $\mathcal F$}\label{Real center}

\begin{prop}\label{Prop real center}
Let $\mathcal F$ be the holomorphic foliation on $\mathbb P^2$ defined in affine coordinates by the logarithmic 1-form
 $$\omega=-\lambda i  \frac{d(y-ix)}{y-ix}+\lambda i \frac{d(y+ix)}{y+ix}+(1+\lambda i )\frac{d(y-3ix)}{y-3ix}
 +(1-\lambda i )\frac{d(y+3ix)}{y+3ix},$$ 
 where $\lambda=\frac{\log 2}{\pi}$. Then $\omega$ defines a global center in $\mathbb R^2$ and there exists a constant $\nu>1$ such that, given $r>0$, the points $(r,0)$ and $(\frac r2,0)$ are contained in the same leaf $L$ of $\mathcal F$ and they are connected by a smooth curve in $L$ of euclidean length smaller than $\nu r$. 
\end{prop} 
\begin{proof}
It can be see that
 $$\omega= 2\lambda\frac{xdy-ydx}{x^2+y^2} +\frac{d(9x^2+y^2)}{9x^2+y^2}
  -6\lambda\frac{xdy-ydx}{9x^2+y^2},
 $$ so $\omega$ defines a real foliation $\mathcal F_{\mathbb R}$ on $\mathbb R^2$ with a unique singularity at the origin. It turns out that the foliation $\mathcal F_{\mathbb R}$ is a global center, that is, its leaves are closed simple curves around the origin. In order to verify this assertion, notice that $\omega$ is invariant by the real homotheties  
 $(x,y)\mapsto (k x, k y)$, $k\in\mathbb R$, so we only need to show that the singularity at the origin is of center type.
  Let $\pi\colon \widetilde{\mathbb C^2}\to\mathbb C^2$ be the blow up at the origin. A complex  line $L$ through $0\in\mathbb C^2$ is strictly transformed by $\pi$ in a line intersecting the exceptional divisor  at a point
 $p_L$. If $L=\mathbb C(a,b)$, $(a,b)\in\mathbb C^2$, we write $p_L=[a:b]$; so we can identify the exceptional divisor with the complex projective line $\mathbb {CP}^1$.  The  strict transform 
 $\widetilde{\mathcal F}$ of $\mathcal F$ by $\pi$ leaves the exceptional divisor invariant and has 4 singularities, all contained in the exceptional divisor,
 at the points $[1:\pm i]$, $[1:\pm 3i]$. Then the closed curve ${\mathbb {RP}^1}=\{[x:y]\in \mathbb {CP}^1\colon x,y\in\mathbb R\}$ avoids the singularities and defines a holonomy map $h\in \diff (\mathbb C, 0)$. We will show that
 $h$ is periodic of period 2, so 
  $\mathcal F_{\mathbb R}$ will be a center.  In order to compute $h$ we can deform ${\mathbb {RP}^1}$ into the 
   circle $C=\{[1:t]\in \mathbb {CP}^1\colon |t-3i|=3\}$. In the $t$-plane $\{[1:t]\in \mathbb {CP}^1\colon t\in\mathbb C\}$, the singularities $[1: i]$ and $[1: 3i]$ are in the interior of $C$, while $[1:- i]$ and $[1:- 3i]$ are in the exterior.  In the coordinates $(t=\frac y x, x)$ of $ \widetilde{\mathbb C^2}$, 
   the circle  $C$ can be positively parametrized in the form $(t(s),0)$, $s\in[0,1]$, $t(0)=t(1)=(0,0)$, and  
   $\widetilde{\mathcal F}$ is defined by 
   \begin{align} \label{holonomy} 2\frac{dx}x=\lambda i\frac{d(t-i)}{t-i}-  \lambda i\frac{d(t+i)}{t+i} 
   -(1+\lambda i)\frac{d(t-3i)}{t-3i} -(1-\lambda i)\frac{d(t+3i)}{t+3i}.
   \end{align}
    Given 
   a point $(0,x_0)$ in $\widetilde{\mathbb C^2}$ with $x_0\in\mathbb C$ small,  the curve $C$ can be lifted to a leaf of $\widetilde{\mathcal F}$, with starting point $(0,x_0)$: this produces a curve of the form $(t(s),x(s))$, $s\in[0,1]$, contained in a leaf of 
   $\widetilde{\mathcal F}$ and such that $x(0)=x_0$. Then $(t(s),x(s))$ satisfies the equation \eqref{holonomy},
  whence
    \begin{align} \int_0^1 2\frac{x'(s)}{x(s)}ds&=\lambda i\int_C\frac{d(t-i)}{t-i}
   -(1+\lambda i)\int_C\frac{d(t-3i)}{t-3i} =-2\pi i.
   \end{align}Therefore $x(1)=-x_0$, which means that $h(z)=-z$ and so $h$ is periodic of order two.
   
   Given $r>0$, let $\alpha_r$ be the orbit of $\mathcal F_{\mathbb R}$ through $(r,0)\in\mathbb R^2$. The orbit
   $\alpha_1$ can be parametrized in the form $\alpha_1(\theta)=\rho(\theta)(\cos\theta,\sin\theta)$, $\theta\in[0,2\pi]$,  for some $ \rho\colon [0,2\pi]\to(0,+\infty)$. 
   Therefore, since $\mathcal F_{\mathbb R}$ is invariant by the homotheties $(x,y)\mapsto (kx,ky)$, we
   can parametrize $\alpha_r(\theta)=r\rho(\theta)(\cos\theta,\sin\theta)$, $\theta\in[0,2\pi]$. In particular
    $\ell (\alpha_r)=r\ell(\alpha_1) $. 
   
 Now, given $r>0$, we will show that the points $(r,0)$ and $(\frac r2,0)$ belong to the same leaf of $\mathcal F$. 
Consider the circle $\tau(s)=i-ie^{is}$, $s\in[0,2\pi]$. Denote by $\eta$ the right hand of equation \eqref{holonomy}.
Then, if we set $f(s)=\frac 12\int_{\tau|_{[0,s]}}{\eta}$, we see that $(t,x)=(\tau (s), re^{f(s)})$ satisfies 
equation \eqref{holonomy}. This means that the curve $(t,x)=(\tau (s), re^{f(s)})$ is contained in a 
leaf of $\widetilde{\mathcal F}$, which in turn means  that the curve $\gamma_r (s)=(re^{f(s)}, \tau (s)re^{f(s)})$, $s\in[0,2\pi]$, is contained in a leaf of
$\mathcal F$. It is straightforward to check that $f(2\pi)=-\pi\lambda=-\log 2$, and so $\gamma_r(2\pi)=(\frac r2,0)$. Thus
$\gamma_r$ connects $f(0)=(r,0)$ with $f(2\pi)=(\frac r2,0)$ in a leaf of $\mathcal F$. Moreover,
if we set $\gamma_1(s)=(e^{f(s)}, \tau (s)e^{f(s)})$ 
it is clear that  $\ell(\gamma_r)=r\ell (\gamma_1)  $. Finally, we fix $\nu>1$ such that $\nu\ge\max\{\ell(\gamma_1),\ell (\alpha_1)\}$, so we have the inequalities  $\ell(\alpha_r),\ell (\gamma_r)<\nu r$ for all $r>0$. \end{proof} 
\begin{remark}\label{roro}
Furthermore,
we can assume $\nu$ so big that $|\rho(\theta)|,|\rho'(\theta)|< \nu$ for all $\theta\in[0,2\pi]$.
\end{remark}

\section{Proof of Theorem \ref{Teorema 2}}
Let $M\subset \mathbb{P}^N$ be a projective 3-manifold and let $\mathcal{D}$ be a singular holomorphic distribution of dimension two. By Theorem   \ref{Teorema 1}, we can assume that $\mathcal D$ is not integrable.  Consider $P\colon  \mathbb{P}^N \dashrightarrow \mathbb{P}^2$ a generic linear projection and $\Pi =P|_M \colon M \dashrightarrow \mathbb{P}^2$ its restriction to $M$, which is dominant and has  generic fibers transverse to $\mathcal{D}$. As explained in Section \ref{sec1},  we can find local holomorphic coordinates $(x,y,z)$   with  $\Pi(x,y,z)=(x,y)$, where $(x,y)$ are affine coordinates on $\mathbb P^2$, such that  $\mathcal{D}$ is defined by a holomorphic 1-form $\Omega= dz - Pdx - Qdy,$
where $P(0)=Q(0)=0$ and $Q_x(0) - P_y(0)=1$.

Now, take the foliation $\mathcal{F}$ of Section \ref{Real center} given  by the 1-form $\omega$ of Proposition \ref{Prop real center} and
let $ \mathcal{F}^{\mathcal{D}}$ its lifting to $M$. It is known that this foliation is generated by a rational  vector field on $M$.  Thus, we must prove that  $ \mathcal{F}^{\mathcal{D}}$  has no 
rational first integral. Furthermore, we will prove that   $ \mathcal{F}^{\mathcal{D}}$   has no meromorphic first integral near the origin. To do so, take  $\delta_0>0$ such that the polydisc
$$\Delta_0=\{|x|<\delta_0, |y|<\delta_0, |z|<\delta_0\}$$
is contained in the domain of our local coordinates an suppose by contradiction that there is a meromorphic first integral  $H$ of 
$ \mathcal{F}^{\mathcal{D}}$   on $\Delta_0$.

The proof is based on the following proposition. If $U$ is an open set in $\mathbb C^3$, we say that $S$ is an analytic surface in $U$ to mean
that $S$ is the zero set of a holomorphic function on $U$. 
\begin{prop}\label{mainprop}
There exists $\delta\in(0,\delta_0)$  such that, if $S$ is an analytic surface in $ \Delta=\{|x|<\delta, |y|<\delta,|z|<\delta\}$ that is 
$\mathcal F ^{\mathcal D}$-invariant and pass through a point of the form
$(r,0,w)$ with $0<r<\delta$ and $|w|<\delta$, then $S$ contains  $\{x=r,y=0\}\cap \Delta$.
\end{prop}

 Let us use this proposition to finish the proof of Theorem \ref{Teorema 2}. Let $\delta>0$  and $\Delta$ be as given by Proposition \ref{mainprop}.    By reducing $\delta>0$ if necessary, we can assume that 
 $|P|,|Q|<\frac{1}{4\nu}$   on  $\Delta$.
  Take a point $\zeta_0=(r,0,z_0)$
 with  $0<r<\frac{\delta}{2\nu}$ and $r+|z_0|<\delta$. The existence of the meromorphic  first integral $H$ allows us to find a  
 $\mathcal F ^{\mathcal D}$-invariant analytic surface $S$ in $\Delta$  passing through $\zeta_0$ and so,
 by Proposition \ref{mainprop}, $S$ contains  $\{x=r,y=0\}\cap \Delta$. Recall that the points
 $(r,0)$ and $(r/2,0)$ are contained in the same leaf $L$ of $\mathcal F$ and there exists
 a curve $\gamma\colon [0,1]\to L$ such that $\gamma(0)=(r,0)$, $\gamma(1)=(r/2,0)$ and
 $\ell (\gamma)< \nu r$. Notice that, since $r<\frac{\delta}{2\nu}$ and $\nu>1$, the curve $\gamma$ is contained
 in the bi-disc $\{|x|<\delta,|y|<\delta\}$. We can lift $\gamma$ to a curve $\tilde\gamma$ tangent to $\mathcal D$
 and starting at the point
 $\zeta_0$; precisely,   if $\gamma(t)=(x(t),y(t))$, the curve $\tilde\gamma$ has the form  $\tilde \gamma(t)=(x(t),y(t),z(t))$, $t\ge 0$ small enough,  where $z(t)$ satisfies
 \begin{align}\nonumber z'=P(x,y,z)x'+Q(x,y,z)y',\quad z(0)=z_0.
 \end{align}
 It is clear that $\tilde\gamma$ is contained in the leaf of  $\mathcal F ^{\mathcal D}$
 through $\zeta_0$. 
 Let us prove that $\tilde\gamma(t)$ is defined and belongs to $\Delta$ for all $t\in[0,1]$. 
  Otherwise
  we can find $\sigma\in(0,1)$ such that $z(t)$ is defined and is contained in $\Delta$  for all $t\in [0,\sigma]$,
  and $\tilde\gamma(\sigma)$ is as close to $\partial \Delta$ as desired. Thus,
  since $\gamma([0,1])$ is contained in  $\{|x|<\delta,|y|<\delta\}$, we can assume that
 $\delta-r/2<|z(\sigma)|<\delta$. Then $|z'(t)|= |Px'+Qy'|\le \frac{1}{2\nu}|\gamma'|$   for all $t\in [0,\sigma]$, so  $|z(\sigma)-z(0)|\le \frac{1}{2\nu}\ell (\gamma)<r/2$ and therefore 
 $|z(\sigma)|<r/2+|z_0|<\delta-r/2$, which is a contradiction.  Then 
 $\tilde\gamma(t)$ is defined and belongs to $\Delta$ for all $t\in[0,1]$ and we conclude that the point
 $(r/2,0,z_1)\colon=\tilde\gamma(1) \in \Delta$  is contained in the leaf of  $\mathcal F ^{\mathcal D}$
 through $\zeta_0$. In particular, $(r/2,0,z_1) \in S$ and, by Proposition \ref{mainprop}, $S$ contains  $\{x=r/2,y=0\}\cap\Delta$.  Proceeding as above, with 1 instead of $\sigma$, we obtain the inequality
 $|z_1|<r/2+|z_0|$, whence $r/2+|z_1|<\delta$. Therefore we can repeat the argument for
 $(r/2,0,z_1)$ instead of $(r,0,z_0)$ to prove that $S$ contains  $\{x=r/4,y=0\}\cap \Delta$. Iterating this process we conclude that $S$ contains  $\{x=r/2^j,y=0\}\cap\Delta$ for all $j\in\mathbb N$
 and therefore $S$ contains the hyperplane $\{y=0\}\cap\Delta$. Then  $\{y=0\}\cap\Delta$
 is $\mathcal F ^{\mathcal D}$-invariant. Therefore, the line  $\{y=0\}$ in the plane $(x,y)$
 is invariant by the foliation $\mathcal F$, which is impossible. \qed

\section{Proof of Proposition \ref{mainprop}}\label{mainsection}
   
   Fix $\epsilon\in (0,1/2)$.  
  From the properties of $\omega$ at the origin we can take $\delta\in (0,\delta_0)$ such that the closure of 
  $$\Delta'\colon =\{|x|<5\nu\delta,|y|<5\nu\delta,|z|<5\nu\delta\} $$ is contained in
$\Delta_0$, and such that  $|P|<\epsilon$, $|Q|<\epsilon$ and
    $|1-\frac{\partial Q}{\partial x}+ \frac{\partial P}{\partial y}|<\epsilon$ on $\overline{\Delta'}$.
   Fix $ M>0$ such that 
   $|\frac{\partial P}{\partial z}|\le M$ and  $|\frac{\partial Q}{\partial z}|\le M$ on $\Delta'$. Notice that, by reducing $\epsilon>0$
   if necessary, we can assume that $\epsilon  M$ is as small as desired --- we use this fact later.
   
  Let  $\gamma\colon[a,b]\to \mathbb C^2$, $\gamma(t)=(x(t),y(t))$,  be a piecewise smooth  curve contained in a leaf of $\mathcal F$. We say that $\tilde\gamma \colon [a,b]\to \mathbb C^3$ is a lifting of $\gamma$ if $\tilde\gamma$  express in the form
  $\tilde\gamma(t)=(x(t),y(t),z(t))$ for some $z \colon [a,b]\to\mathbb C$ piecewise smooth,  such that
  $$z'=P(\tilde\gamma)x'+Q(\tilde\gamma)y'.$$ In this case $\tilde\gamma$ is contained
   in a leaf of $\mathcal F_{\mathcal D}$. If $\tilde\gamma$ is contained in $\Delta'$, it follows from the transversality of $\mathcal D$ to the fibers of $\Pi\colon (x,y,z)\to (x,y)$ that $\tilde\gamma$ is the unique lifting of 
   $\gamma$ with starting point $\tilde\gamma(a)$. 
  
   \begin{lemma}\label{levantar} Consider  $(r,0,w)$ with $0<r<\delta$ and $|w|<\delta$. Let $\gamma\colon[a,b]\to \mathbb C^2$  be a piecewise smooth  curve starting at $(r,0)$ and contained in a leaf of $\mathcal F$, such that $\ell (\gamma)<4\nu r$.  Then $\gamma$ can be lifted to a curve $\tilde\gamma$ starting at $(r,0,w)$ and contained in $\Delta'$.  
   \end{lemma}
   \proof We assume that $\gamma$ is smooth; the general case is easily obtained from this case. 
    Since $\nu>1$,  we see that $\gamma$ is contained in the bi-disc $\{|x|<5\nu\delta, |y|<5\nu\delta\}$.
 If $\gamma$ is parametrized as $\gamma(t)=(x(t),y(t))$ for $t\in [0,1]$, its lifting $\tilde\gamma$ starting at $(r,0,w)$,
 if it exists, 
  has the form  $\tilde \gamma(t)=(x(t),y(t),z(t))$, where $z(t)$ satisfies
 \begin{align} z'=P(x,y,z)x'+Q(x,y,z)y',\quad z(0)=w.
 \end{align}
 By the transversality of $\mathcal D$ to the fibers of $\Pi$, the lifting $\tilde\gamma$ do exists at least on a small interval $[0,\varepsilon]$, $\varepsilon >0$.
 Let us prove that $\tilde\gamma(t)$ is defined and belongs to $\Delta'$ for all $t\in[0,1]$. 
 Otherwise there exists $\sigma\in(0,1)$ such that $\tilde\gamma(t)$ is defined and is contained in $\Delta'$ 
  for all $t\in [0,\sigma]$ and $\tilde\gamma(\sigma)$ is as close to $\partial \Delta'$ as desired. Thus,
  since $\gamma([0,1])$ is contained in  $\{|x|<5\nu\delta,|y|<5\nu\delta\}$,  we can assume that 
 $(4\nu+1)\delta<|z(\sigma)|<5\nu\delta$. Then
  $|z'(t)|= |Px'+Qy'|\le 2\epsilon|\gamma'|$   for all $t\in [0,\sigma]$, so 
  $|z(\sigma)|\le |z(0)| + 2\epsilon\ell (\gamma)<\delta+
  8\epsilon\nu\delta<(4\nu+1)\delta$, which is a contradiction.\qed \\

   Given $r>0$, let $\alpha_r\colon[0,2\pi]\to\mathbb R^2$ be the orbit of the center $\mathcal F_{\mathbb R}$ through $(r,0)$, as defined in  Section \ref{Real center}.  
    Given $w\in\mathbb C$, let $\alpha_{r,w}$ be  --- if it is defined --- the lifting of $\alpha_r$ starting
   at $(r,0,w)$. Thus, if $\alpha_r(\theta)=(x_r(\theta), y_r(\theta))$, we have  $\alpha_{r,w}(\theta)=(x_r(\theta), y_r(\theta), \zeta_{r,w}(\theta))$, $\theta\in[0,2\pi]$, where
   $\zeta_{r,w}(0)=w$ and 
    \begin{align}\label{eqdif}\zeta_{r,w}'=
    P(\alpha_{r,w})x_r'
    +Q(\alpha_{r,w})y'_r.
    \end{align} 
   \begin{lemma}\label{cuadratico}There exists a constant $c>1$ such that, given $(r,0,w)$  with
   $0<r<\delta$ and $|w|<\delta$, then $\alpha_{r,w}$ is defined,  is contained in $\Delta'$ and $$\frac 1 c r^2\le|\zeta_{r,w}(2\pi)-\zeta_{r,w}(0)|
   \le cr^2.$$
   \end{lemma}
   
   \proof Fix $(r,0,w)$ with
   $0<r<\delta$ and $|w|<\delta$. From Section \ref{Real center} we have that $\ell(\alpha_r)<\nu r$. Thus, by Lemma \ref{levantar} we see that $\alpha\colon =\alpha_{r,w}$ is defined and contained in $\Delta'$. By simplicity, we write $\zeta=\zeta_{r,w}$. 
   Let $\eta$ be the euclidean  segment joining  $\alpha(2\pi)=(r,0,\zeta(2\pi))$ to $\alpha(0)=(r,0,w)$. Then
   $\alpha*\eta $ is a closed curve and we have
   $$  \zeta(2\pi)-\zeta(0)=\int_{\alpha}Pdx+Qdy= \int_{\alpha*\eta}Pdx+Qdy.$$
   Let $\mathcal C$ be the compact cone with vertex $O=(0,0,w)$ and guiding curve $\alpha*\eta $, that is, $\mathcal C$ is generated by
   the segments $[O,p]$ as $p$ runs along  $\alpha*\eta $. Then, it follows from  Stokes' Theorem that
   \begin{align}\label{ko0} 
   \zeta(2\pi)-\zeta(0)=\int_{\mathcal C}\left(Q_x-P_y\right)dxdy
   +\int_{\mathcal C}P_zdzdx +\int_{\mathcal C}
   Q_zdzdy.
   \end{align}
   We can express  
   $\mathcal C=\mathcal C_\alpha\cup \mathcal C_\eta$, where  $\mathcal C_\alpha$ and 
   $\mathcal C_\eta$ are the compact cones  with vertex $O=(0,0,w)$ and guiding curves $\zeta $ and $\eta $, respectively.  Recall --- from Section \ref{Real center} --- that 
   $\alpha_r(\theta)=(r\rho(\theta)\cos\theta,r\rho(\theta)\sin\theta)$. Then $\alpha(\theta)=
   (r\rho(\theta)\cos\theta,r\rho(\theta)\sin\theta,\zeta(\theta))$ and \eqref{eqdif} becomes 
    \begin{align}\label{eqdif1}\zeta'=
    rP(\alpha)(\rho\cos\theta)'
    +rQ(\alpha)(\rho\sin\theta)'.
    \end{align} 
   Now,
   $\mathcal C_\alpha$ can be parametrized as 
   \begin{align}\nonumber
   (x,y,z)=(tr\rho\cos\theta,tr\rho\sin\theta,(1-t)w+t\zeta(\theta));\quad (t,\theta)\in[0,1]\times
   [0,2\pi].
   \end{align}
   Thus, a straightforward computation shows that 
   \begin{align*}dxdy&=r^2\rho^2tdtd\theta;\\
   dzdx&=  [(\zeta-w)(r\rho\cos\theta)'-\zeta'r\rho\cos\theta]tdtd\theta;     \\
   dzdy&= [(\zeta-w)(r\rho\sin\theta)'-\zeta'r\rho\sin\theta]tdtd\theta     .
   \end{align*}
   Notice that $\mathcal C_\eta$ is the triangle of vertices $(0,0,w)$,  $(r,0,w)$  and  $(r,0,\zeta(2\pi))$, hence 
   $dxdy|_{\mathcal C_\eta}=dzdy|_{\mathcal C_\eta}=0$ and so
   $ \int_{\mathcal C}\left(Q_x-P_y\right)dxdy=\int_{\mathcal C_\alpha}\left(Q_x-P_y\right)dxdy$
   and $ \int_{\mathcal C}Q_zdzdy=\int_{\mathcal C_\alpha}Q_zdzdy$. Then,
   on the one hand,
    \begin{align}
    \int_{\mathcal C}\left(Q_x-P_y\right)dxdy&
    =\left(\int_{[0,1]\times[0,2\pi]}\left(Q_x-P_y\right)\rho^2 tdtd\theta\right) r^2
    \end{align} 
    and,
since  
   $|(Q_x-P_y)-1|<\epsilon<\frac 12 $ on $\Delta'$, we find positive constants  $c_0$ and $ c_1$ independent of $\epsilon$ such that 
    \begin{align}\label{ko1}
   c_0r^2\le \left|\int_{\mathcal C }\left(Q_x-P_y\right)dxdy\right|\le c_1 r^2.
   \end{align} 
  On the other hand, 
   \begin{align}\nonumber
   \int_{\mathcal C} Q_zdzdy=
   \int_{[0,1]\times[0,2\pi]} Q_z \left[(\zeta-w)(r\rho\sin\theta)'-\zeta'r\rho\sin\theta\right]tdtd\theta.
   \end{align}
   Since $|\rho|,|\rho'|<\nu$ --- see Remark \ref{roro}, we have $|(\rho\sin\theta)'|,|(\rho\cos\theta)'|\le 2\nu$ and it follows from \eqref{eqdif1} that
   $|\zeta'(\theta)|\le 4\epsilon \nu r $ and $|\zeta(\theta) - w|\le 8\pi\epsilon \nu r $ for all $\theta\in[0,2\pi]$. Therefore
     \begin{align}\begin{aligned}\label{ko2}
   \left|\int_{\mathcal C} Q_zdzdy\right|&\le
   \int_{[0,1]\times[0,2\pi]} |Q_z|(16\pi+4)\epsilon\nu^2r^2|tdtd\theta|\\
   &\le (16\pi+4)\epsilon M\nu^2\left(\int_{[0,1]\times[0,2\pi]} tdtd\theta\right) r^2.
   \end{aligned}
   \end{align}
  Proceeding as above we also obtain the bound
   \begin{align}\label{ko3}
   \left|\int_{\mathcal C_\alpha} P_zdzdx\right|
   &\le (16\pi+4)\epsilon M\nu^2\left(\int_{[0,1]\times[0,2\pi]} tdtd\theta\right) r^2.
   \end{align} Thus, in order to estimate $ \int_{\mathcal C}P_zdzdx = \int_{\mathcal C_\alpha}P_zdzdx 
   + \int_{\mathcal C_\eta}P_zdzdx $, it remains to deal with the integral
   $ \int_{\mathcal C_\eta}P_zdzdx $. Notice that $\mathcal C_\eta$ is parametrized
   as $$(x,y,z)=(s r,0, w+t [\zeta(2\pi)-w]),$$ where $(s,t)$ runs on the triangle $T=\{(s,t)\in\mathbb R^2\colon 
   0\le t\le s\le 1\}.$
   Then 
     \begin{align}\nonumber
   \left|\int_{\mathcal C_\eta} P_zdzdx\right|
   &= \left| \int_{T}  P_z [\zeta(2\pi)-w]rdtds\right|\le \int_{T}|P_z|(8\pi\epsilon \nu r)rdsdt,
   \end{align}
   so that
     \begin{align}\label{ko4}
   \left|\int_{\mathcal C_\eta} P_zdzdx\right|
   \le 4\pi\epsilon M \nu r^2.
   \end{align}
   Finally, using in \eqref{ko0} the bounds
    \eqref{ko1} --- \eqref{ko4}, if $\epsilon M$ is small enough, we find a constant $c>1$ such that
   $\frac 1 c r^2\le |\zeta(2\pi)-\zeta(0)|\le cr^2$. \qed
   
  \subsection*{Proof of Proposition \ref{mainprop}}
   Take a point $(r,0,w)$ with $0<r<\delta$ and $|w|<\delta$.
   By the properties of $\mathcal F$, the  point $(r,0)$ can be connected with $(\frac r 2, 0)$
   by a curve $\beta_1$ of length $\ell(\beta_1)<\nu r$, which is contained in the leaf $L$
   of $\mathcal F$ through $(r,0)$. Again, the point $(\frac r 2, 0)$ is connected with
   $(\frac r 4, 0)$  by a curve $\beta_2$ of length $\ell(\beta_2)<\nu (\frac r 2)$ contained in  $L$. And so on. Thus, given $n\in\mathbb N$, the curve $\beta_1 *\dots *\beta_n$  connects $(r,0)$ 
   with  $(\frac r {2^n}, 0)$ in $L$. Let $\alpha_n :=\alpha_{\frac r {2^n}}$ be the orbit of the center 
   $\mathcal F_{\mathbb R}$ passing through $(\frac r {2^n},0)$ ---
 see section \ref{Real center}.  Then the loop 
 $$\gamma_n \colon= \beta_1 *\dots *\beta_n*\alpha_n*\beta_n^{-1} *\dots *\beta_1^{-1}$$ is based at $(r,0)$ and contained in $L$. Since
   $$\ell (\gamma_n)=2\sum_{j=1}^n \ell(\beta_j) + \ell (\alpha_n)<
   2\nu\sum_{j=1}^n \frac {r}{2^{j-1}}+ \nu\frac r{2^n}<4\nu r, $$
   by Lemma \ref{levantar} we have that $\gamma_n$ can be lifted to a curve $\tilde\gamma_n$ starting at $(r,0,w)$, which is contained in the leaf $ L$ of $\mathcal F^\mathcal D|_{\Delta'}$ through $(r,0,w)$.   If $(r,0,w_n)$ is the ending point of  $\tilde\gamma_n$,
 we will prove that $w_n\neq w$ and that $\lim_{n\to\infty} w_n=w$. If we do so, Proposition
 \ref{mainprop} will be proved: If $S$ is a $\mathcal F ^{\mathcal D}$-invariant analytic surface in $\Delta$  passing through $(r,0,w)$, then $S$ contains $L$ and so $(r,0,w_n)\in S$ for all $n\in\mathbb N$,
 whence $S$ contains the line  $\{x=r,y=0\}\cap\Delta$. 
Let us prove that
  $w_n\neq w$. Let $(\frac r {2^n},0,u)$ be the ending point of the lifting of 
   $\beta_1 *\dots *\beta_n$ starting at $(r,0,w)$ and let $(\frac r {2^n},0,v)$ be the ending point of the lifting of $\alpha_n$ starting at  $(\frac r {2^n},0,u)$ . Then $(r,0,w_n)$ is the ending point of the lifting of 
   $\beta_n^{-1} *\dots *\beta_1^{-1}$ starting at  $(\frac r {2^n},0,v)$ or, equivalently,
   $(\frac r{2^n},0,v)$ is the ending point of the lifting of  $\beta_1 *\dots *\beta_n$ starting at $(r,0,w_n)$. Thus,   
   if $w_n$ were equal to $w$, by the unicity of the lifting, $v$ would be equal to $u$, which is impossible because, by Lemma \ref{cuadratico}, we have $|v-u|>\frac1c \left(\frac r{2^n}\right)^2$.   Let $\beta(t)=(x(t),y(t))$, $t\in[0,1]$, be a parametrizacion of  $\beta_n^{-1} *\dots *\beta_1^{-1}$. Let $\beta_u(t)=(x(t), y(t),z_u(t))$ be the lifting of $\beta$ starting at 
    $(\frac r {2^n},0,u)$  and let $\beta_v(t)=(x(t), y(t),z_v(t))$ be the lifting of $\beta$ starting at  
    $(\frac r {2^n},0,v)$. Thus  $z_u'=P(\beta_u)x'+Q(\beta_u)y'$ and  $z_v'=P(\beta_v)x'+Q(\beta_v)y'$
    and,  if we set $f(t)=z_v(t)-z_u(t)$,  $$f'=[P(\beta_v)-P(\beta_u)]x'+[Q(\beta_v)-Q(\beta_u)]y'.$$
    Then, since $|P(\beta_v)-P(\beta_u)|$ and $|Q(\beta_v)-Q(\beta_u)|$ are bounded by $M|z_v-z_u|$,
    we have that 
    $|f'|\le2M |\beta'||f|$. Therefore 
    $|f(t)|\le |f(0)|e^{\int_0^t2M|\beta'(s)|ds}$ and, since $\ell (\beta)<2\nu r$, we see that
       $|f(t)|\le e^{4M \nu r}|f(0)|.$ In particular, for $t=1$, we obtain that
       $|w_n-w|\le e^{4M\nu  r}|v-u|$ and therefore, by Lemma \ref{cuadratico}, we conclude that
        $|w_n-w|\le ce^{4M\nu  r}\left(\frac r{2^n}\right)^2$, which shows that $\lim w_n= w$.\qed


\begin{thebibliography}{99}
\frenchspacing



\bibitem{Ar}
{\sc V.I. Arnold}
\emph{Chapitres Supplémentaires de la Théorie des équations différentielles ordinaires.}
Éditions Librairie du Globe, 1996. 

\bibitem{Bon}
{\sc P. Bonnet}
\emph{Families of k-derivations on k-algebras.}
J. Pure Appl. Algebra 199 (2005), no.1-3, 11–26.

\bibitem{Faro}
{\sc M. Falla Luza, R. Rosas}
\emph{Distributions, first integrals and Legendrian foliations.}
Bull. Braz. Math. Soc. 53 (2022), no.4, 1157–1229.

\bibitem{NagNow}
{\sc M. Nagata, A. Nowicki}
\emph{Rings of constants for k-derivations in $k[x_1,\ldots, x_n]$.}
J. Math. Kyoto Univ.28(1988), no.1, 111–118.

\bibitem{Now}
{\sc A. Nowicki}
\emph{Rings and fields of constants for derivations in characteristic zero.}
J. Pure Appl. Algebra 96 (1) (1994) 47–55.



\end{thebibliography}
\end{document}